\documentclass{article}

\newcommand{\heuteIst}{May 24, 2011 }

\usepackage{amsmath,amsfonts,latexsym,amscd,amssymb,enumerate,amsthm}

\usepackage[backref=page]{hyperref}
 \usepackage[abbrev]{amsrefs}

\swapnumbers \theoremstyle{plain} 

\newtheorem{theorem}{Theorem}

\newtheorem{proposition}[theorem]{Proposition}

\theoremstyle{definition} 
\newtheorem{definition}[theorem]{Definition}

\theoremstyle{remark}

\newcommand{\complexs}{\mathbb{C}} 
\newcommand{\naturals}{\mathbb{N}} 
\newcommand{\integers}{\mathbb{Z}} 
\newcommand{\rationals}{\mathbb{Q}}

\DeclareMathOperator{\id}{id} 
 
\DeclareMathOperator{\Mat}{M}

\newcommand{\abs}[1]{\left\lvert#1\right\rvert} 

\newcommand{\onto}{\twoheadrightarrow}

\newcommand{\subgroup}{\leq}

\newcommand{\normalsubgroup}{\lhd}

\DeclareMathOperator{\lcm}{lcm}

\DeclareMathOperator{\tr}{tr} 
\DeclareMathOperator{\pr}{pr}

\newcommand{\forget}[1]{} 
 
\newcommand{\innerprod}[1]{\langle #1 \rangle}

{\catcode`@=11\global\let\c@equation=\c@theorem}

\allowdisplaybreaks[2]

\begin{document} \pagestyle{myheadings} \markboth{Peter Linnell, Boris Okun, Thomas Schick}{Atiyah conjecture for Coxeter groups}

\date{Last compiled \today; last edited \heuteIst or later}

\title{The strong Atiyah conjecture for right-angled Artin and Coxeter groups}

\author{Peter Linnell\thanks{ \protect\href{mailto:plinnell@math.vt.edu}{e-mail: plinnell@math.vt.edu}\protect\\
\protect\href{http://www.math.vt.edu/people/plinnell/}{www: http://www.math.vt.edu/people/plinnell/}\protect\\
partially supported by a grant from the NSA.} \\
Math.~Dept.\\Virginia Tech\\
Blacksburg, VA 24061-0123, USA 
\and 
Boris Okun\thanks{\protect\href{mailto:okun@uwm.edu}{e-mail: okun@uwm.edu} \protect\\
\protect\href{http://www.uwm.edu/\string~okun}{www:~http://www.uwm.edu/\kern -.15em\lower .85ex\hbox{\~{}}\kern .04em okun} 
}
\\
Department of Mathematical Sciences\\
University of Wisconsin--Milwaukee\\
Milwaukee, WI 53201, USA 
\and 
Thomas Schick\thanks{ \protect\href{mailto:schick@uni-math.gwdg.de}{e-mail: schick@uni-math.gwdg.de} \protect\\
\protect\href{http://www.uni-math.gwdg.de/schick}{www:~http://www.uni-math.gwdg.de/schick} \protect\\
partially funded by the Courant Research Center ``Higher order structures in Mathematics" within the German initiative of excellence }\\
Mathematisches Institut\\
Georg-August-Universit\"at G{\"o}ttingen\\
Germany\\
} 
\maketitle
\begin{abstract}
	We prove the strong Atiyah conjecture for right-angled Artin groups and right-angled Coxeter groups.
	More generally, we prove it for groups which are certain finite extensions or elementary amenable extensions of such groups.
\end{abstract}

When Atiyah introduced $L^2$-Betti numbers for compact manifolds (later generalized to finite CW-complexes), he asked \cite{Atiyah}*{p.~72} about the possible values these can assume.
In particular, he asked whether they are always rational numbers, or even integers if the fundamental group is torsion free.
This question was later popularized in precise form as ``the strong Atiyah conjecture''. Easy examples show that the possible values depend on the fundamental group $G$ of the space in question.
For a subgroup of order $n$ in $G$, a compact manifold with $\pi_1(M)=G$ and
with $L^2$-Betti number $1/n$ can easily be constructed.
For certain groups $G$ which contain finite subgroups of arbitrarily large order, with quite some effort manifolds $M$ with $\pi_1(M)=G$ and with transcendental $L^2$-Betti numbers have been constructed \cites{Austin,Grabowski,Pichot-Schick-Zuk}.

The $L^2$-Betti numbers can be computed from the cellular chain complex.
The chain groups there are of the form $l^2(G)^d$, and the differentials are given by convolution multiplication with a matrix over $\integers[G]$.
Let $\mathbb{N}$ denote the positive integers.
Then the strong Atiyah conjecture for CW-complexes with fundamental group $G$ turns out to be equivalent to the following (with $K=\integers$):
\begin{definition}\label{DAtiyah}
	Let $G$ be a group with a bound on the orders of finite
subgroups and let $\lcm(G)\in\naturals$ denote the least common multiple of these orders.
	Let $K\subset\complexs$ be a subring.
	
	We say that $G$ satisfies the \emph{strong Atiyah conjecture over $K$, or $KG$ satisfies the strong Atiyah conjecture} if for every $n\in\naturals$ and every $A\in \Mat_n(KG)$ 
	\begin{equation*}
		\dim_G(\ker(A)):=\tr_G(\pr_{\ker A}) = \sum_{i=1}^n \innerprod{\pr_{\ker A} \delta_1 e_i,\delta_1 e_i}_{l^2G^n} \ \in \frac{1}{\lcm(G)}\integers . 
	\end{equation*}
	Here, as before, we consider $A\colon l^2(G)^n\to l^2(G)^n$ as a bounded operator, acting by left convolution multiplication --- the continuous extension of the left multiplication action on the group ring to $l^2(G)$.
	We denote by $\pr_{\ker A}$ the orthogonal projection onto the kernel.
	Finally, $\delta_1 e_i:= (0,\dots,0,\delta_1,0,\dots,0)^t\in l^2(G)^n$, where $^t$ denotes transpose, is the standard basis element with non-zero entry in the $i$-th column equal to the characteristic function of the neutral element in $l^2G$.
	
	If $G$ contains arbitrarily large finite subgroups, we set $\lcm(G):=+\infty$.
\end{definition}

This conjecture has several important consequences.
The most notable is the Kaplansky conjecture: if $G$ is torsion free (hence $\lcm(G)=1$) the strong Atiyah conjecture over $K$ implies that $KG$ contains no non-trivial zero divisors.
For groups with torsion, knowledge of the strong Atiyah conjecture can still be useful e.g.~to show that $\ker(A)=\{0\}$, it suffices to show that its dimension is $<1/\lcm(G)$.
In particular, as explained in \cite{MR1812434}*{Section 12}, for certain right-angled Coxeter groups, knowledge of the strong Atiyah conjecture can be used to prove the Singer vanishing conjecture.

There is a considerable body of work to establish the strong Atiyah conjecture for suitable classes of groups.
In this paper, we use these results to establish the Atiyah conjecture for certain elementary amenable extensions of right-angled Artin groups, right-angled Coxeter groups and related groups. 

This is based on the following results: 
\begin{enumerate}
	\item \label{item:solvable}
	With a very ring-theoretic approach to extensions by an infinite cyclic groups and induction methods, Linnell \cite{MR1242889} proves the Atiyah conjecture if $G$ is elementary amenable and $\lcm(G)<\infty$, in particular if $G$ is torsion-free solvable, and for arbitrary coefficient rings $K\subset\complexs$.
	\item \label{item:res_solvable}
	In \cite{MR1777117} and more generally in \cite{MR1990479}*{Theorem 1.4} there is established an approximation result for $L^2$-Betti numbers for residually elementary amenable groups in terms of the $L^2$-Betti numbers for these elementary amenable quotients.
	Because a limit of a sequence of integers has to be an integer, this
        implies the strong Atiyah conjecture over $\overline{\rationals}$ for
        residually torsion-free elementary amenable. 
	Here $\overline{\rationals}$ is the field of algebraic
        numbers. Recall that the class of elementary amenable groups is the
        smallest class of groups which contains all finite and abelian groups
        and is closed under extensions, subgroups and directed union. In
        particular, all solvable groups are elementary amenable.
	\item \label{item:linnell_schick}
	Linnell and Schick establish in \cite{MR2328714}*{Theorem 4.60} a
        method which implies that the strong Atiyah conjecture not only holds for a
        group $G$, but also for its elementary amenable extensions, in
        particular all its finite extensions. 
	It is based on the use of Galois cohomology (even a version of the generalized cohomology theory ``stable cohomotopy'') which governs the appearance of excessive torsion in the quotients of the extensions in question.
	\end{enumerate}

The main result of this note is  the following: 
\begin{theorem}
	\label{theo:atiyah_for_coxeter}
	Let $H$ be a right-angled Artin group or the commutator subgroup of a right-angled Coxeter group, and let 
	\begin{equation*}
		1\to H\to G\to Q\to 1 
	\end{equation*}
	be an extension with $Q$ elementary amenable and such that $\lcm(G)<\infty$.
	In the case $H$ is the commutator subgroup of a right-angled Coxeter group, assume that all finite subgroups of $Q$ are $2$-groups.
	Then the group $G$ satisfies the strong Atiyah conjecture over $\overline{\rationals}$.
	
	Note that a right-angled Coxeter group itself is such an extension (with $Q$ a finite 2-group) of its commutator subgroup.
\end{theorem}

To prove this, we will observe that an old result of Duchamp and Krob
\cite{MR1179860} shows that right-angled Artin groups are residually torsion
free nilpotent, and therefore belong to the class \ref{item:res_solvable}
above. 
Moreover, this  result and a recent result of Lorensen
\cite{MR2561762} allows us to apply the method of \ref{item:linnell_schick}.

For right-angled Coxeter groups, we will combine a construction of Davis and Januszkiewicz \cite{MR1783167} with a slight generalization of \ref{item:linnell_schick}.

\begin{definition}
	\label{def:artincoxeter}
	Let $L$ be a flag simplicial complex with the vertex set $S$.
	Associated to $L$ there are three groups: the right-angled Artin group $A_L$, the right-angled Coxeter group $W_L$, and its commutator subgroup $C_{L}$ defined by 
	\begin{align*}
		A_L& :=\langle S \mid st=ts \;\text{ if $\{s,t\}$ is an edge in }L\rangle,\\
		W_L &:=\langle S\mid st=ts \;\text{ if $\{s,t\}$ is an edge in } L, \text{ and }s^2=1 \text{ for } s\in S\rangle, \\
		C_{L}&:=[W_{L},W_{L}]. 
	\end{align*}
\end{definition}
If $S'$ is another copy of $S$, we denote the element corresponding to $s \in S$ by $s' \in S'$.
Let $L'$ be a copy of $L$ with vertex set $S'$.
The octahedralization $OL$ of $L$ is the subcomplex of the join of $L$ and $L'$ obtained by removing all the edges (and their cofaces) of the form $\{s,s'\}$.
We denote by $\Delta$ the simplex with vertex set $S'$, and define $\Delta L$ to be the subcomplex of the join of $L$ and $\Delta$ obtained by removing all the edges (and their cofaces) of the form $\{s,s'\}$.
Note that the Coxeter group associated to $\Delta$ is $W_{\Delta}=(\mathbb Z/2)^{\abs{S}}$.
We identify $W_{L}$ with the subgroups generated by $S$ in $W_{OL}$ and $W_{\Delta L}$.

Following \cite{MR1783167}, we define four homomorphisms:
\begin{align*}
	\phi\colon W_{\Delta L} &\to W_{\Delta} & \theta\colon W_{\Delta L} &\to W_{\Delta} & \alpha\colon W_{OL } &\to W_{\Delta L} & \beta\colon A_L &\to W_{\Delta L}\\
	s &\mapsto 1 & s &\mapsto s' & s &\mapsto s & s &\mapsto ss'\\
	s' &\mapsto s' & s' &\mapsto s' & s' &\mapsto s'\negthinspace ss' 
\end{align*}
\begin{theorem}[{\cite{MR1783167}*{Theorem on p.~230}\label{thm:dj}}] 
	The maps $\alpha \colon W_{OL }\to \ker(\phi)$ and $\beta \colon A_L\to \ker(\theta)$ are isomorphisms, and we identify $W_{OL }$ and $A_L$ with their images.
	Thus, both $W_{OL }$ and $A_L$ are normal subgroups of $W_{\Delta L}$
        of index $2^{\abs{S}}$.
\end{theorem}
We need the following properties and relations between the groups just
defined. If $H\subgroup G$ is a subgroup, recall that a \emph{retraction}
$\pi\colon W\to H$ is a 
homomorphism with $\pi|_H=\id_H$, and $H$ is called \emph{retract} of $W$ in
this case.
\begin{proposition}
	\label{prop:groups}~
	\begin{enumerate}
		\item\label{item:CDL}
		$C_{\Delta L}=W_{OL }\cap A_L$ is a normal subgroup both in $A_L$ and $W_{OL }$ of index $2^{\abs{S}}$.
		\item\label{item:CL}
		$C_L= W_L\cap A_L$.
		\item\label{item:retrL}
		$C_{L}$ is a retract of $C_{\Delta L }$.
		\item\label{item:finite}
		$A_{L}$ and $C_{L}$ have finite classifying spaces.
		\item\label{item:lcm}
		$\lcm(A_{L})=\lcm(C_{L})=1$, $\lcm(W_{L})=2^{\dim{L}+1}$.
	\end{enumerate}
\end{proposition}
\begin{proof}
	The direct product of the homomorphisms $\phi$ and $\theta$ maps $W_{\Delta L}$ into $(\mathbb{Z}/2)^{2\abs{S}}$ and is easily seen to be the abelianization map.
	Thus \ref{item:CDL} follows from Theorem \ref{thm:dj}.
	Similarly \ref{item:CL} follows, since $\theta$ restricts to the abelianization map on $W_{L}$.
	To prove \ref{item:retrL} we note that the natural retraction $W_{\Delta L} \to W_{L}$, $ s \mapsto s$, $s' \mapsto 1$ restricts to the retraction of the commutator subgroups.
	\ref{item:finite} and \ref{item:lcm} are standard facts, see, for example, \cite{MR2360474}.
\end{proof}

We now have to establish the very slight generalization of the method of \cite{MR2328714} alluded to in \ref{item:linnell_schick}.
To begin with, recall the following \cite{MR2328714}*{Definition 4.3}.
\begin{definition}
	Let $p$ be a prime number.
	A discrete group $G$ is called \emph{cohomologically $p$-complete} if the canonical homomorphism 
	\begin{equation*}
		\hat H^*(\hat G^p;\integers/p\integers):=\lim_{[G:H]=p^n} H^*(G/H;\integers/p\integers) \to H^*(G;\integers/p\integers) 
	\end{equation*}
	is an isomorphism.
	Here, the direct limit is taken over the directed system of finite $p$-group quotients of $G$.
	This limit is by definition the Galois cohomology of the pro-$p$-completion $\hat G^p$ of $G$.
	
	$G$ is called \emph{cohomologically complete} if it is cohomologically $p$-complete for every $p$.
\end{definition}
Slightly specializing \cite{MR2328714}*{Definition 4.53}, we also define 
\begin{definition}
	\label{def:enough_quotients}
	\emph{$G$ has enough torsion-free quotients for the prime $p$} if for each normal subgroup $V$ of $G$ with $G/V$ a finite $p$-group, there exists a subgroup $U\subgroup V $ normal in $G$ with $G/U$ torsion-free and elementary amenable.
	If this holds for every $p$, we say that \emph{$G$ has enough torsion-free quotients}.
\end{definition}
Since subgroups of torsion-free elementary amenable groups are also torsion-free and elementary amenable, the above definition is equivalent to saying that every epimorphism from $G$ onto a finite $p$-group factors through a torsion-free elementary amenable group.

The next proposition shows that these properties are inherited by normal subgroups of $p$-power index, as well as retracts.
\begin{proposition}
	\label{prop:main_of_Linnell_S}
	Let $p$ be a prime number and let $H$ be a subgroup of $G$.
	\begin{enumerate}
		\item \label{item:cohom_complete_to_subgrp}
		Suppose $H$ is normal and the index of $H$ in $G$ is a power of $p$.
		If $G$ is cohomologically $p$-complete, then $H$ is cohomologically $p$-complete.
		If $G$ has enough torsion-free quotients for the prime $p$, then so has $H$.
		\item\label{item:retract}
		Suppose that $H$ is a retract of $G$.
		If $G$ is cohomologically $p$-complete, then $H$ is cohomologically $p$-complete.
		If $G$ has enough torsion-free quotients for the prime $p$, then so has $H$.
	\end{enumerate}
\end{proposition}
\begin{proof}
	The first assertion of \ref{item:cohom_complete_to_subgrp} is \cite{MR2328714}*{Proposition 4.16}.
	
	The second assertion follows from \cite{MR2328714}*{Lemma 4.11 (2)}: if $V\normalsubgroup H$ with $H/V$ a finite $p$-group, then the intersection $V_{0}$ of all $G$-conjugates of $V$ is a subgroup of $H$ normal in $G$ with $H/V_0$, and therefore $G/V_0$, a finite $p$-group.
	Thus, by assumption, there exists a subgroup $U$ of $V_{0}$ normal in $G$ with $G/U$ torsion-free elementary amenable.
	Then $U$ is normal in $H$, and $H/U$ is torsion-free and elementary amenable.
	
	To prove \ref{item:retract}, observe that the Galois cohomology of the pro-$p$-completion, as well as the usual group cohomology, are contravariant functors.
	Therefore, since $H$ is a retract of $G$, the comparison map between $H^n(H;\integers/p)$ and $\hat H^n(\hat H^{p};\integers/p)$ is a direct summand of the comparison map between $H^n(G;\integers/p)$ and $\hat H^n(\hat G^{p};\integers/p)$.
	Thus, if the latter is an isomorphism, then so is the former.
	
	For the torsion-free quotients, let $q\colon H\onto B$ be an epimorphism onto a finite $p$-group $B$ and let $\pi\colon G\to H$, $\pi|_{H}=\id_H$ denote the retraction.
	Then, by assumption, the composition $G\xrightarrow{\pi}H\to B$ factors through a torsion-free elementary amenable group.
	Since the composition agrees with $q$ on $H$, the restriction of this factorization to $H$ gives the desired factorization of $q$.
\end{proof}
\newpage
\begin{proposition}\nopagebreak ~
	\label{prop:prop_of_C}\nopagebreak
	\begin{enumerate}
		\item\label{item:Artin}
		A right-angled Artin group $A_{L}$ is residually torsion-free nilpotent, cohomologically complete, and has enough torsion-free quotients.
		
		\item\label{item:comm}
		The commutator subgroup $C_L$ of a right-angled Coxeter group $W_L$ is residually torsion-free nilpotent, cohomologically $2$-complete, and has enough torsion-free quotients for the prime $2$.
	\end{enumerate}
\end{proposition}
\begin{proof}
	By \cite{MR1179860}*{Theorem 2.1} the terms of the lower central series of $A_L$ intersect in the identity and have torsion-free factors.
	Thus $A_{L}$ is a residually torsion-free nilpotent group, and, since by Proposition \ref{prop:groups}\ref{item:CL} $C_{L}\subgroup A_{L}$, so is $C_{L}$.
Also, because infinitely many terms of the lower central series of
$A_L$ have torsion-free factors and because finite $p$-groups are
nilpotent, it follows that $A_L$ has enough torsion-free quotients,
cf.~\cite{MR2328714}*{Example 4.56}.
Furthermore $A_L$ is cohomologically complete by \cite{MR2561762}*{Theorem 2.6}.
	Therefore, by Propositions \ref{prop:groups}\ref{item:CDL} and \ref{prop:main_of_Linnell_S}\ref{item:cohom_complete_to_subgrp}, $C_{\Delta L }$ is cohomologically $2$-complete and has enough torsion-free quotients for the prime $2$.
	Finally, by Propositions \ref{prop:groups}\ref{item:retrL} and \ref{prop:main_of_Linnell_S}\ref{item:retract} its retract $C_L$ also has these properties.
\end{proof}
	
\begin{proposition}
	\label{prop:AC}
	Assume the following: 
	\begin{itemize}
		\item $1\to H\to G\to Q\to 1$ is an exact sequence of groups 
		\item $H$ has a finite classifying space, is cohomologically $p$-complete and has enough torsion-free quotients at the prime $p$.				
		\item $K\subset \mathbb{C}$ is a subfield fixed by complex conjugation and $KH$ satisfies the strong Atiyah conjecture 
		\item $Q$ is elementary amenable and every finite subgroup of $Q$ is a $p$-group 
		\item $\lcm(G) < \infty$.
	\end{itemize}
	Then $KG$ also satisfies the strong Atiyah conjecture.

	Note that the last two conditions are satisfied if $Q$ is a finite $p$-group.
\end{proposition}
\begin{proof}
	This proposition corresponds to the assertions of \cite{MR2328714}*{Theorem 4.60, Corollary 4.62}.
	The only difference is that instead of cohomological $q$-com\-pleteness for all primes $q$ we assume that all finite subgroups of $Q$ are $p$-groups.
	However, the proof of \cite{MR2328714}*{Theorem 4.60} is done one prime at a time and indeed establishes exactly the statement made here.
\end{proof}

Finally, we are ready to prove the main result.
\begin{proof}
	[Proof of Theorem \ref{theo:atiyah_for_coxeter}.] By Proposition \ref{prop:prop_of_C} $A_L$ and $C_L$ are residually torsion-free nilpotent groups and therefore satisfy the strong Atiyah conjecture over $\overline{\mathbb{Q}}$.
	
	Now the theorem follows immediately from \cite{MR2328714}*{Theorem 4.60, Corollary 4.62} and Proposition \ref{prop:AC}, as their remaining assumptions  are satisfied by Proposition \ref{prop:prop_of_C} and Proposition \ref{prop:groups}\ref{item:finite}.
\end{proof}


\begin{bibdiv}
  \begin{biblist}
\bib{Atiyah}{article}{
   author={Atiyah, M. F.},
   title={Elliptic operators, discrete groups and von Neumann algebras},
   conference={
      title={Colloque ``Analyse et Topologie'' en l'Honneur de Henri Cartan
      (Orsay, 1974)},
   },
   book={
      publisher={Soc. Math. France},
      place={Paris},
   },
   date={1976},
   pages={43--72. Ast\'erisque, No. 32-33},
   review={\MR{0420729 (54 \#8741)}},
}

\bib{Austin}{unpublished}{
  author={Austin, Tim},
  title={Rational group ring elements with kernels having irrational
    dimension},
  date={2009},
  note={arXiv:0909.2360},
}


\bib{MR2360474}{book}{
   author={Davis, Michael W.},
   title={The geometry and topology of Coxeter groups},
   series={London Mathematical Society Monographs Series},
   volume={32},
   publisher={Princeton University Press},
   place={Princeton, NJ},
   date={2008},
   pages={xvi+584},
   isbn={978-0-691-13138-2},
   isbn={0-691-13138-4},
   review={\MR{2360474 (2008k:20091)}},
}
		
\bib{MR1783167}{article}{
   author={Davis, Michael W.},
   author={Januszkiewicz, Tadeusz},
   title={Right-angled Artin groups are commensurable with right-angled
   Coxeter groups},
   journal={J. Pure Appl. Algebra},
   volume={153},
   date={2000},
   number={3},
   pages={229--235},
   issn={0022-4049},
   review={\MR{1783167 (2001m:20056)}},
   doi={10.1016/S0022-4049(99)00175-9},
}

\bib{MR1812434}{article}{
      author={Davis, Michael~W.},
      author={Okun, Boris},
       title={Vanishing theorems and conjectures for the {$\ell^2$}-homology of
  right-angled {C}oxeter groups},
        date={2001},
        ISSN={1465-3060},
     journal={Geom. Topol.},
      volume={5},
       pages={7\ndash 74},
         url={http://dx.doi.org/10.2140/gt.2001.5.7},
      review={\MR{1812434 (2002e:58039)}},
}


\bib{MR1990479}{article}{
   author={Dodziuk, J{\'o}zef},
   author={Linnell, Peter},
   author={Mathai, Varghese},
   author={Schick, Thomas},
   author={Yates, Stuart},
   title={Approximating $L^2$-invariants and the Atiyah conjecture},
   note={Dedicated to the memory of J\"urgen K. Moser},
   journal={Comm. Pure Appl. Math.},
   volume={56},
   date={2003},
   number={7},
   pages={839--873},
   issn={0010-3640},
   review={\MR{1990479 (2004g:58040)}},
   doi={10.1002/cpa.10076},
}

\bib{MR1179860}{article}{
   author={Duchamp, G.},
   author={Krob, D.},
   title={The lower central series of the free partially commutative group},
   journal={Semigroup Forum},
   volume={45},
   date={1992},
   number={3},
   pages={385--394},
   issn={0037-1912},
   review={\MR{1179860 (93e:20047)}},
   doi={10.1007/BF03025778},
}
		
\bib{Grabowski}{unpublished}{
  author={Grabowski, Lukasz},
  title={On Turing machines, dynamical systems and the Atiyah problem},
  date={2010},
  note={arXiv:1004.2030}
}

\bib{MR1242889}{article}{
   author={Linnell, Peter A.},
   title={Division rings and group von Neumann algebras},
   journal={Forum Math.},
   volume={5},
   date={1993},
   number={6},
   pages={561--576},
   issn={0933-7741},
   review={\MR{1242889 (94h:20009)}},
   doi={10.1515/form.1993.5.561},
}
		
\bib{MR2328714}{article}{
   author={Linnell, Peter},
   author={Schick, Thomas},
   title={Finite group extensions and the Atiyah conjecture},
   journal={J. Amer. Math. Soc.},
   volume={20},
   date={2007},
   number={4},
   pages={1003--1051 (electronic)},
   issn={0894-0347},
   review={\MR{2328714 (2008m:58041)}},
   doi={10.1090/S0894-0347-07-00561-9},
}

\bib{MR2561762}{article}{
   author={Lorensen, Karl},
   title={Groups with the same cohomology as their pro-$p$ completions},
   journal={J. Pure Appl. Algebra},
   volume={214},
   date={2010},
   number={1},
   pages={6--14},
   issn={0022-4049},
   review={\MR{2561762 (2010k:20088)}},
   doi={10.1016/j.jpaa.2009.04.002},
}
		
\bib{Pichot-Schick-Zuk}{unpublished}{
  author={Pichot, Mika\"el},
  author={Schick, Thomas },
  author={Zuk, Andrzej},
  title={Closed manifolds with transcendental L2-Betti numbers},
  year={2010},
  note={arXiv:1005.1147},
}

\bib{MR1777117}{article}{
   author={Schick, Thomas},
   title={Integrality of $L^2$-Betti numbers},
   journal={Math. Ann.},
   volume={317},
   date={2000},
   number={4},
   pages={727--750},
   issn={0025-5831},
   review={\MR{1777117 (2002k:55009a)}},
   doi={10.1007/PL00004421},
  note={\textit{Erratum} in vol.~\textbf{322}, 421--422}
}




			  
  \end{biblist}
\end{bibdiv}
\end{document}